\documentclass[preprint,12pt]{elsarticle}

\usepackage{graphicx}
\usepackage{subcaption}
\usepackage{amssymb, amsmath}
\usepackage{amsthm}
\usepackage{color}
\usepackage{array}
\usepackage{tabu}

\journal{International Journal of Non-Linear Mechanics}

\newtheorem{theorem}{Theorem}[section]
\newtheorem{lemma}[theorem]{Lemma}

\theoremstyle{definition}

\theoremstyle{remark}
\newtheorem{remark}[theorem]{Remark}

\begin{document}

\begin{frontmatter}


\title{Averaging method and asymptotic solutions in some mechanical problems}



\author[label5]{Ivan Polekhin}
\address[label5]{Steklov Mathematical Institute of the Russian Academy of Sciences, Moscow, Russia}


\begin{abstract}
In the paper we consider systems in oscillating force fields such that the classical method of averaging can be applied. We present sufficient conditions for the existence of forced oscillations in such systems and study asymptotic behaviour of some solutions. In particular, we show that for an inverted pendulum with a horizontally moving pivot point in an oscillating gravity field there exists a solution along which the pendulum never falls.

\end{abstract}


\end{frontmatter}


\section{Introduction}
\label{S:1}

One of the approaches to study the dynamics of mechanical systems is to construct some continuous map --- that is easier to understand than the original equations --- and then use it to prove the existence of the required properties in the system. Let us illustrate these general words with an example. Let us have a smooth vector field on $\mathbb{R}^2$ (Fig. 1):

\begin{align}
\begin{split}
\label{eq1}
    &\dot x = v(x, y),\\
    &\dot y = w(x, y).
\end{split}
\end{align}
We suppose that all solutions of this system can be continued for all $t \in \mathbb{R}$. We also suppose that for $y = 1$ we have $w(x,1)>0$ and for $y=-1$ we always have $w(x,-1) < 0$. Then in the strip $y \in (-1,1)$ there exists a solution such that $y(t) \in (-1,1)$ for all $t \geqslant 0$. Indeed, let us consider a segment $x = 0$, $y \in [-1, 1]$ and show that there exists a point in this segment such that the solution starting at this point satisfies the above property. Suppose from the contrary that any solution starting at this segment leaves the strip at some $t \geqslant 0$, i.e. $|y(t)| = 1$. Since the vector field is transverse to the boundary, then we can construct a continuous map between $[-1,1]$ and the boundary of the strip, that is defined by the flow of the system. In other words, if some solution leaves our strip through the part of the boundary where $y = 1$, then all solutions with close initial data also leave our region through the same component of the boundary. It follows from the continuous dependence on the initial data. The boundary has two connected components and the segment is connected. The contradiction proves the statement.

\begin{figure}[!h]
\centering
\begin{minipage}[t]{180px}
  \centering
  \def\svgwidth{180px}\footnotesize
\begingroup%
  \makeatletter%
  \providecommand\color[2][]{%
    \errmessage{(Inkscape) Color is used for the text in Inkscape, but the package 'color.sty' is not loaded}%
    \renewcommand\color[2][]{}%
  }%
  \providecommand\transparent[1]{%
    \errmessage{(Inkscape) Transparency is used (non-zero) for the text in Inkscape, but the package 'transparent.sty' is not loaded}%
    \renewcommand\transparent[1]{}%
  }%
  \providecommand\rotatebox[2]{#2}%
  \newcommand*\fsize{\dimexpr\f@size pt\relax}%
  \newcommand*\lineheight[1]{\fontsize{\fsize}{#1\fsize}\selectfont}%
  \ifx\svgwidth\undefined%
    \setlength{\unitlength}{457.78192692bp}%
    \ifx\svgscale\undefined%
      \relax%
    \else%
      \setlength{\unitlength}{\unitlength * \real{\svgscale}}%
    \fi%
  \else%
    \setlength{\unitlength}{\svgwidth}%
  \fi%
  \global\let\svgwidth\undefined%
  \global\let\svgscale\undefined%
  \makeatother%
  \begin{picture}(1,0.61848836)%
    \lineheight{1}%
    \setlength\tabcolsep{0pt}%
    \put(0,0){\includegraphics[width=\unitlength,page=1]{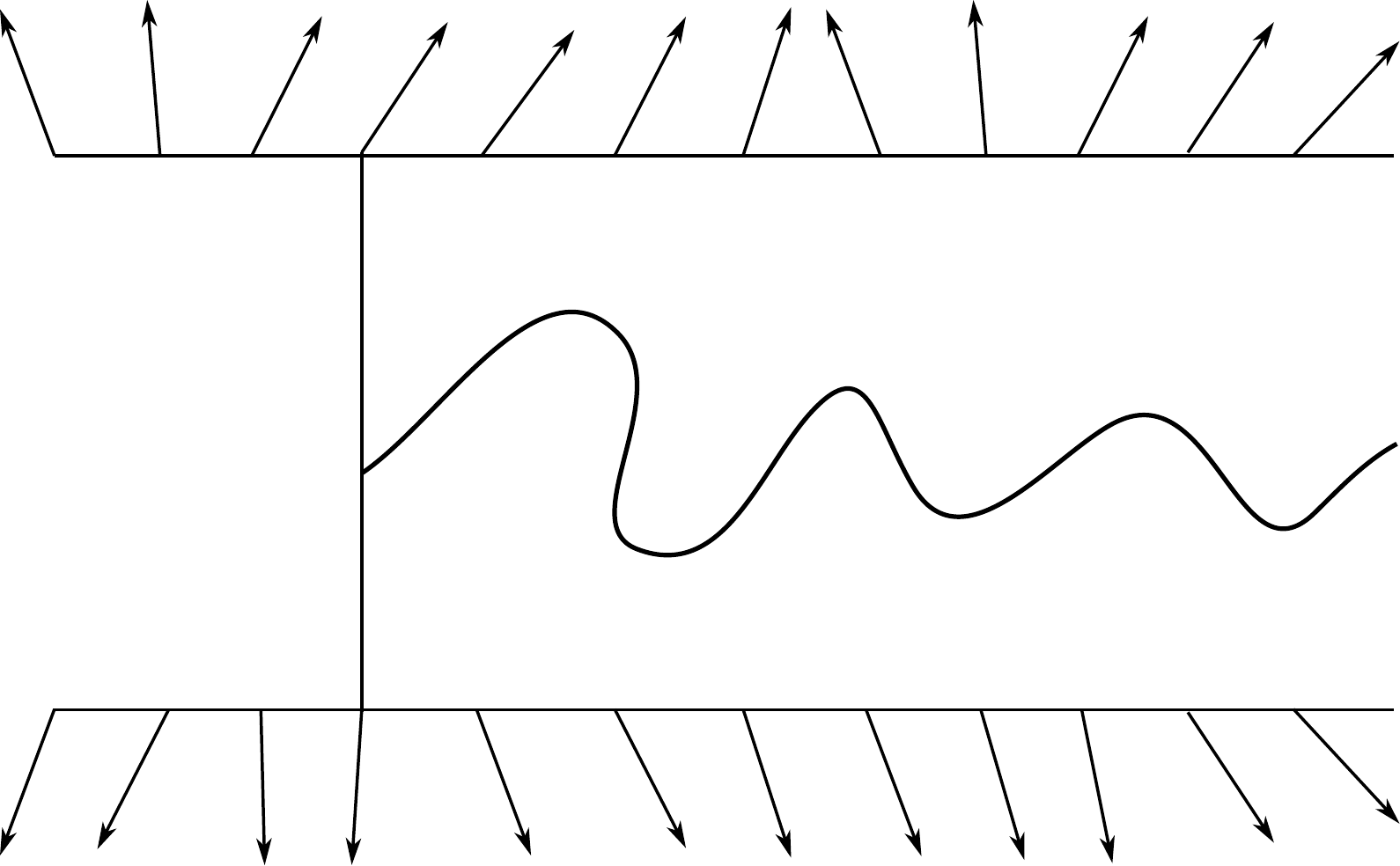}}%
    \put(0.04550089,0.14207025){\color[rgb]{0,0,0}\makebox(0,0)[lt]{\lineheight{1.25}\smash{\begin{tabular}[t]{l}$y=-1$\end{tabular}}}}%
    \put(0.04617746,0.44415131){\color[rgb]{0,0,0}\makebox(0,0)[lt]{\lineheight{1.25}\smash{\begin{tabular}[t]{l}$y=1$\end{tabular}}}}%
    \put(0.26597392,0.2169656){\color[rgb]{0,0,0}\makebox(0,0)[lt]{\lineheight{1.25}\smash{\begin{tabular}[t]{l}$x=0$\end{tabular}}}}%
  \end{picture}%
\endgroup%

    \caption{Schematic presentation of the vector field and solution for which $y(t) \in (-1,1)$.}
    \label{pic3}
\end{minipage}
\hspace{0.1cm}
\begin{minipage}[t]{180px}
  \centering
  \def\svgwidth{180px}\footnotesize
\begingroup%
  \makeatletter%
  \providecommand\color[2][]{%
    \errmessage{(Inkscape) Color is used for the text in Inkscape, but the package 'color.sty' is not loaded}%
    \renewcommand\color[2][]{}%
  }%
  \providecommand\transparent[1]{%
    \errmessage{(Inkscape) Transparency is used (non-zero) for the text in Inkscape, but the package 'transparent.sty' is not loaded}%
    \renewcommand\transparent[1]{}%
  }%
  \providecommand\rotatebox[2]{#2}%
  \newcommand*\fsize{\dimexpr\f@size pt\relax}%
  \newcommand*\lineheight[1]{\fontsize{\fsize}{#1\fsize}\selectfont}%
  \ifx\svgwidth\undefined%
    \setlength{\unitlength}{457.78192677bp}%
    \ifx\svgscale\undefined%
      \relax%
    \else%
      \setlength{\unitlength}{\unitlength * \real{\svgscale}}%
    \fi%
  \else%
    \setlength{\unitlength}{\svgwidth}%
  \fi%
  \global\let\svgwidth\undefined%
  \global\let\svgscale\undefined%
  \makeatother%
  \begin{picture}(1,0.61848836)%
    \lineheight{1}%
    \setlength\tabcolsep{0pt}%
    \put(0,0){\includegraphics[width=\unitlength,page=1]{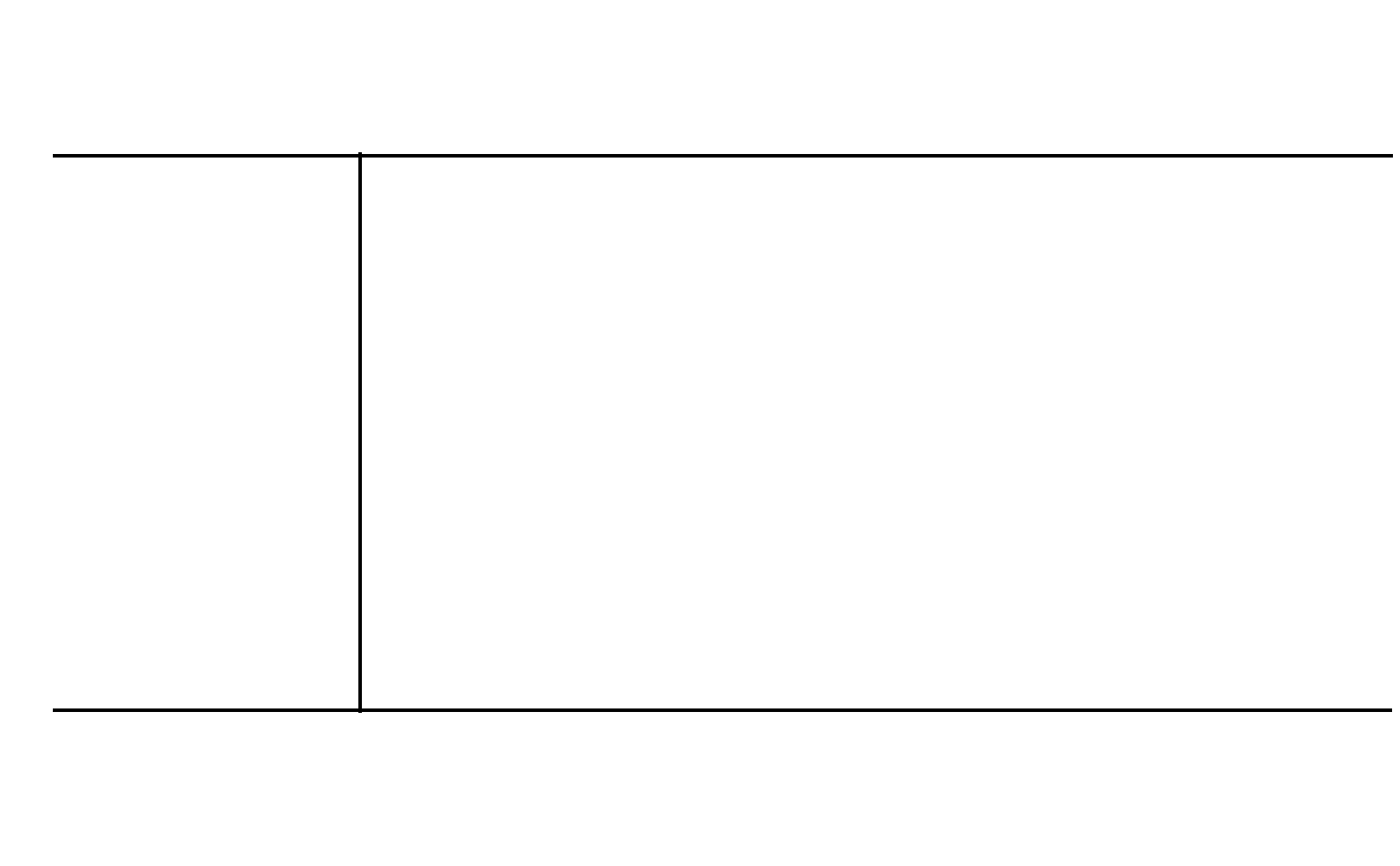}}%
    \put(0.0444739,0.1419079){\color[rgb]{0,0,0}\makebox(0,0)[lt]{\lineheight{1.25}\smash{\begin{tabular}[t]{l}$y=-1$\end{tabular}}}}%
    \put(0.04515046,0.44398895){\color[rgb]{0,0,0}\makebox(0,0)[lt]{\lineheight{1.25}\smash{\begin{tabular}[t]{l}$y=1$\end{tabular}}}}%
    \put(0.26494693,0.21680324){\color[rgb]{0,0,0}\makebox(0,0)[lt]{\lineheight{1.25}\smash{\begin{tabular}[t]{l}$x=0$\end{tabular}}}}%
    \put(0,0){\includegraphics[width=\unitlength,page=2]{fig2a.pdf}}%
  \end{picture}%
\endgroup%

    \caption{The map between the segment and the boundary may become discontinuous if there is an internally tangent solution.}
    \label{pic4}
\end{minipage}
\end{figure}

The above consideration is a simple form of the Wa{\.z}ewski topological method \cite{wazewski1947principe,reissig1963qualitative}. The key element in our proof is the possibility to construct a continuous map to the boundary. For this, it is important for the vector field to be transverse or to be externally tangent to the boundary. Indeed, if there is a solution that is internally tangent to the boundary, the map defined by the flow can be discontinuous (Fig. 2).

Applying the Wa{\.z}ewski method to the system describing an inverted pendulum with a horizontally moving pivot point, it is possible to show that, for a given law of motion of the pivot, there always exists a solution such that the pendulum never falls \cite{polekhin2014examples,polekhin2014periodic}. The same is true for an inverted spherical pendulum.

Based on a theorem from \cite{srzednicki2005fixed}, one can show the existence of a periodic solution without falling \cite{srzednicki2019periodic,polekhin2014examples,polekhin2014periodic, polekhin2015forced,polekhin2016forced}. A stronger result holds for Lagrangian systems \cite{bolotin2015calculus}: there exist several periodic solutions, one for each homotopy class of freely homotopic loops on the configuration space with boundary. The key property used in the proof is the dynamical convexity of the corresponding boundary. In other words, all tangent solutions should be externally tangent.

In the paper, we show how to apply the Wa{\.z}ewski method when the region is not dynamically convex. We will consider specific mechanical systems, however, the results can be generalised easily. In particular, we consider an inverted pendulum with a horizontally moving pivot point  in an oscillating gravity field and the system of a rotating closed curve with a mass point sliding along this curve, also in an oscillating gravity field.

\section{A simple problem}

Let us consider the following modification of system (\ref{eq1}):
\begin{align*}
\begin{split}
     &\dot x = v(x,y),\\
    &\dot y = w(x,y) + \lambda u(x, y, t), \quad \lambda \in \mathbb{R}, \quad |u(x,y,t)| = 1.
\end{split}
\end{align*}
Here and everywhere below all functions are $C^\infty$ smooth. Suppose that $w(x,1)>a\geqslant0$ and $w(x,-1)<-a$ for all $x$. If $|\lambda| \leqslant a$, then we can use the same considerations as above and conclude that there exists a solution that remains in the strip $y \in (-1,1)$ for all $t \geqslant 0$.

For $|\lambda| > a$ the corresponding map to the boundary can be discontinuous and the Wa{\.z}ewski method cannot be applied directly. However, we will show that a result similar to the above one can be proved provided the non-autonomous perturbation satisfies some additional assumptions. Let us now have a system:

\begin{align}
\begin{split}
\label{eq2}
    &\dot x = v(x,y),\\
    &\dot y = w(x,y) + u(\lambda t), \quad \lambda \in \mathbb{R}.
\end{split}
\end{align}
Here we suppose that $u(t)$ is a $T$-periodic function with zero mean value

$$
\bar u = \frac{1}{T}\int_0^T u(t)\, dt = 0.
$$
Below we will use the following version of Bogolyubov's theorem \cite{bogolyubov1961asymptotic,mitropolsky1967averaging}.

\begin{theorem}
Let us have a system
\begin{align}
\label{eq33}
    \dot x = \varepsilon X(t, x), \quad x \in \mathbb{R}^n.
\end{align}
Suppose that $X(t,x)$ satisfies the following conditions
\begin{enumerate}
    \item For some domain $D \subset \mathbb{R}^n$ there exist $M$ and $\mu$ such that for all $t \geqslant 0$ and $x,x',x'' \in D$ we have
    $$
    \| X(t,x) \| \leqslant M, \quad \| X(t, x') - X(t, x'') \| \leqslant \mu \| x' - x'' \|.
    $$
    \item Uniformly on $x$ in $D$
    $$
    \lim_{T \to \infty} \frac{1}{T} \int\limits_0^T X(t,x)\, dt = X_0(x).
    $$
\end{enumerate}
Then for any arbitrarily small $\rho$ and $\delta$ and for arbitrarily large $L$ there exists $\varepsilon_0$ such that, if $x_0(t)$ is a solution of

$$
\dot x_0 = \varepsilon X_0(x_0),
$$
defined for $t_0 < t < \infty$, $t_0 \geqslant 0$ and remains in $D$ with the $\rho$-neighbourhood, then for $0 < \varepsilon < \varepsilon_0$ in $t_0 < t < t_0 + L/\varepsilon$ we have

$$
\| x(t) - x_0(t) \| < \delta,
$$
where $x(t)$ is the solution of the initial equation and $x(t_0) = x_0(t_0)$.
\end{theorem}
\begin{remark}
As usual, we can consider the change of time variable $t \mapsto t/\varepsilon = \lambda t$ and rewrite (\ref{eq33}) as
$$
\dot x = X(\lambda t, x).
$$
\end{remark}
Let us prove now that in (\ref{eq2}), for $\lambda$ large enough, there exists a solution that always remains in $y \in (-1,1)$. To be more precise, the following result holds.
\begin{theorem}
Let functions $v(x,y)$, $w(x,y)$ and $u(t)$ satisfy the following conditions:
\begin{enumerate}
    \item There exists positive $\delta$ such that $v(x,y)$ and $w(x,y)$ and their first derivatives are bounded for $y \in (-\delta + 1, 1 + \delta) \cup (-\delta - 1, - 1 + \delta)$ and for all $x$.
    \item There exists positive $a$ such that
    $$
    w(x,1) > a, \quad w(x,-1) < -a, \quad\mbox{for all}\quad x \in \mathbb{R}.
    $$
    \item Function $u(t)$ is $T$-periodic and the mean value $\bar u$ is zero.
\end{enumerate}
Suppose also that all solutions of (\ref{eq2}) starting in the region where $|y| \leqslant 1$ can be continued for all $t \geqslant 0$. Then there exists a solution $(x(t),y(t))$ of (\ref{eq2}) for which $|y(t)| < 1$ for all $t \geqslant 0$.
\end{theorem}
\begin{proof}
There exists a small number $\delta_1 \in (0, \delta)$ such that for every solution $(\tilde x(t), \tilde y(t))$ of system (\ref{eq1}) with the initial data $\tilde x(t_0) = x_0$, $\tilde y(t_0) = y_0 \in [-\delta_1 + 1, 1]$, $t_0 \geqslant 0$ we have $\tilde y(t_0 + \tau) \geqslant 1 + \delta_1 + \varepsilon_1$ for some $\tau > 0$ and $\varepsilon_1 > 0$, one for all considered solutions. It follows from the first two conditions of the theorem. The same is true for $y = -1$, i.e. for any solution with initial data $x_0$, $y_0 \in [\delta_1 - 1, -1]$ we have $\tilde y(t_0 + \tau) \leqslant -1 - \delta_1 - \varepsilon_1$.

Let us now consider the following modified system
\begin{align}
\begin{split}
\label{eq3}
    &\dot x = v(x,y),\\
    &\dot y = w(x,y) + \sigma(y)\cdot u(\lambda t),
\end{split}
\end{align}
where $\sigma(y)$ has the form:
$$
\sigma(y)=
\begin{cases}
& 0, \quad y \in [-1-\delta_1/2,-1 + \delta_1/2] \cup [1-\delta_1/2,1 + \delta_1/2],\\
& 1, \quad y \in (-\infty, -1-\delta_1] \cup [-1 +\delta_1, 1 - \delta_1] \cup [1+\delta_1, +\infty),\\
&\mbox{monotone in all other intervals}.
\end{cases}
$$
In other words, for the modified system (\ref{eq3}) the non-autonomous perturbation $u(\lambda t)$ is zero in the vicinity of lines $y = \pm 1$. Therefore, for this system there exists an initial condition $x(0) = 0$ and $y(0) = y_0 \in (-1,1)$ such that the corresponding solution $(x(t), y(t))$ always remains in the region where $y \in (-1,1)$.

Let us now show that the same is true for the original system for large $\lambda > 0$. From Theorem 2.1 we have that for large $\lambda > 0$ for any solution $(\tilde x(t), \tilde y(t))$ of (\ref{eq1}) and for any solution $(x(t), y(t))$ of (\ref{eq3}) we obtain
$$
| \tilde x(t) - x(t) | + |\tilde y(t) - y(t)| < \varepsilon_1, \quad t \in [t_0, t_0 + \tau]
$$
provided $\tilde x(t_0) = x(t_0)$ and $\tilde y(t_0) = y(t_0) = y_0 \in [-1, -1 + \delta_1] \cup [1 - \delta_1, 1]$. Suppose that for the solution of (\ref{eq3}) that always remains in $y \in (-1,1)$, for some $t_0\geqslant 0$, we have $y(t_0) \in [-1,-1+\delta_1] \cup [1,1+\delta_1]$. Then we have $y(t_0 + \tau) > 1$ or $y(t_0 + \tau) < -1$. From this contradiction we obtain that solution $(x(t),y(t))$ cannot reach the region in which the original equation (\ref{eq2}) is modified.
\end{proof}

In the conclusion of the section, we would like to outline once again the idea of the proof. Let us have a system such that we can prove --- by means of the Wa{\.z}ewski method --- the existence of a solution that always remains in some region. Let us now have a perturbed system such that the Wa{\.z}ewski method cannot be applied because some solutions become externally tangent to the boundary of the region. However, we assume that the perturbation is oscillating and a theorem on (local) averaging can be applied in a vicinity of the boundary. Next we locally modify our vector field such that the Wa{\.z}ewski method can be applied to the modified system. At the same time, from the averaging, we obtain that all solutions, starting at the vicinity of the boundary where the system is modified, leave the considered region. Therefore, we obtain the following. First, from the Wa{\.z}ewski method, we have a solution of the modified system that always remains in the considered region. Second, this solution cannot be in the region where the system is modified, since in this case, the solution leaves the region. Finally, this solution of the modified problem never goes through the modified vicinity of the boundary and we conclude that the same solution exists in the original perturbed problem.

\section{Main results}
In this section we consider several mechanical systems for which we will prove results similar to Theorem 2.3 and also will study the existence of periodic solutions.

Let us consider the following equation
\begin{align}
\label{eq5}
    \ddot x = f(t) \sin x - \cos x.
\end{align}
This equation describes the motion of a pendulum in a gravity field in the presence of a non-autonomous horizontal force. We can interpret (\ref{eq5}) as the equation of motion for a pendulum with a horizontally moving pivot point.

Now suppose that the gravity field is also non-autonomous and oscillates
\begin{align}
\label{eq6}
    \ddot x = f(t) \sin x - (1 + g(\lambda t))\cos x.
\end{align}
\begin{theorem}
Suppose that $f(t)$ and $\dot f(t)$ are bounded, $g(t)$ is periodic and $\bar g = 0$. Then there exists $\lambda_0$ such that, for all $\lambda \geqslant \lambda_0$, equation (\ref{eq6}) has a solution $x(t)$ satisfying $x(t) \in (0, \pi)$ for all $t > 0$.
\end{theorem}
\begin{proof}
First, let $\delta > 0$ be such that any solution $\tilde{x}(t)$ of (\ref{eq5}) starting at $\tilde x_0 = \tilde x(t_0) \in [0, \delta] \cup [\pi-\delta, \pi]$, $\dot{\tilde{x}}_0 = \dot{\tilde{x}}(t_0) \in [-\delta, \delta]$ satisfies the following condition. There exist $\varepsilon >0$, $\tau > 0$, same for all considered initial conditions $t_0$, $\tilde{x}_0$ and $\dot{\tilde{x}}_0$, such that $\tilde{x}(t_0 + t') = -\varepsilon$ or $\tilde{x}(t_0 + t') = \pi + \varepsilon$ for some $t' \in [0, \tau]$. The existence of such $\delta$ follows from continuous dependence of solutions from the right-hand side for system (\ref{eq5}). Therefore, for $x$ close to $0$, the solutions of (\ref{eq5}) are close to the solutions of $\ddot x = -1$. Similarly, for $x$ close to $\pi$, we have equation $\ddot x = 1$.

Now consider a modified version of equation (\ref{eq6}):
\begin{align}
\label{eq7}
    \ddot x = f(t) \sin x - (1 + \sigma(x,\dot x)g(\lambda t))\cos x.
\end{align}
Here $\sigma(x, \dot x)$ is a smooth function with bounded derivatives such that
$$
\sigma(x, \dot x)=
\begin{cases}
& 0, \quad x \in [-\delta/2,\delta/2] \cup [\pi-\delta/2,\pi + \delta/2] \mbox{ and } \dot x \in [-\delta/2,\delta/2],\\
& 1, \quad x \not\in [-\delta,\delta] \cup [\pi-\delta,\pi + \delta] \mbox{ or } \dot x \not\in [-\delta,\delta],\\
&\mbox{smooth elsewhere}.
\end{cases}
$$
Applying the Wa{\.z}ewski method, it can be easily shown \cite{polekhin2014examples,polekhin2014periodic} that for any $\lambda$ equation (\ref{eq7}) has a solution $x(t)$ satisfying $x(t) \in (0, \pi)$ for all $t > 0$. Equations (\ref{eq6}) and (\ref{eq7}) differs only for $x \in [-\delta,\delta] \cup [\pi-\delta,\pi + \delta]$ and $\dot x \in [-\delta,\delta]$.

From Theorem 2.1 we have that for large $\lambda > 0$ for any solution $\tilde{x}(t)$ of (\ref{eq5}) and for any solution $x(t)$ of (\ref{eq7}) we have
$$
|x(t) - \tilde{x}(t)| < \varepsilon/2, \quad |\dot x(t) - \dot{\tilde{x}}(t)| < \varepsilon/2, \quad t \in [t_0, t_0 + \tau]
$$
provided $\tilde x(t_0) = x(t_0) = x_0 \in [0 ,\delta] \cup [\pi - \delta, \pi]$ and $\dot{\tilde{x}}(t_0) = x(t_0) = \dot x_0 \in [-\delta, \delta]$. Suppose that for the solution $x(t)$ of (\ref{eq7}) that always remains in $x \in (0,\pi)$, for some $t_0\geqslant 0$, we have $x(t_0)  \in [0 ,\delta] \cup [\pi - \delta, \pi]$ and $x(t_0) = \in [-\delta, \delta]$. Then we have $x(t') > \pi$ or $x(t') < 0$ for some $t' > t_0$. From this contradiction we obtain that solution $x(t)$ cannot reach the region in which the original equation (\ref{eq6}) is modified.
\end{proof}

\begin{remark}
In the above considerations we omitted the part containing the proof of the fact that equation (\ref{eq7}) has a solution $x(t)$ such that $x(t) \in (0, \pi)$ for $t > 0$. However, the same method can be applied to the next example, for which we consider the proof in detail.
\end{remark}

Let us have a smooth strictly convex closed curve in a vertical plane. Consider a point moving along this curve without any friction. Suppose that the curve is rotating in the vertical plane around a fixed point $O$. We will denote the angle of rotation by $\varphi$. For any $\varphi$, there are two points on the curve ($L(\varphi)$ and $R(\varphi)$ correspondigly) where the tangent line is vertical. Since the curve is strictly convex, these points depend smoothly on the angle $\varphi$. Moreover, we assume that for some $c$
\begin{align}
\label{phi_cond}
|\dot \varphi(t)| < c, \quad |\ddot \varphi(t)| < c \quad\mbox{for all }t.
\end{align}
Let $s$ define the natural parametrization of the curve. We denote the coordinates of the curve w.r.t. coordinate frame $O\xi\eta$ by $\xi(s)$ and $\eta(s)$. Below we suppose that the mass of the point and the gravity acceleration equal $1$. The kinetic energy has the form:
\begin{align*}
    T = \frac{1}{2}\left( \dot s^2 + \dot\varphi^2(\xi^2 + \eta^2) + 2\dot\varphi\dot s (\xi\eta' - \eta\xi') \right).
\end{align*}
The potential has the form:
\begin{align*}
    V = - (\xi\sin\varphi + \eta \cos\varphi).
\end{align*}
The dynamics of the system is described by the following equation:
\begin{align}
\label{eq8}
    \ddot s + \ddot\varphi(\xi\eta' - \eta\xi') - \dot\varphi^2(\xi\xi' + \eta\eta') +  (\xi'\sin\varphi + \eta' \cos\varphi) = 0.
\end{align}
For any $t$, there exist two points $s_1(t)$ ($L(\varphi)$) and $s_2(t)$ ($R(\varphi)$) such that
\begin{align*}
    \begin{split}
        & \xi' (s_1(t)) \sin\varphi(t) + \eta'(s_1(t)) \cos\varphi(t) = -1,\\
        & \xi' (s_2(t)) \sin\varphi(t) + \eta'(s_2(t)) \cos\varphi(t) = 1.
    \end{split}
\end{align*}
We suppose that $s_1(t) < s_2(t)$. For a given function $\varphi(t)$, functions $s_1(t)$ and $s_2(t)$ are smooth and determined uniquely. Let us consider the following modified equation
\begin{align}
\label{eq9}
    \ddot s + \ddot\varphi(\xi\eta' - \eta\xi') - \dot\varphi^2(\xi\xi' + \eta\eta') +  (1 + \sigma(t, s, \dot s)g(\lambda t))(\xi'\sin\varphi + \eta' \cos\varphi) = 0.
\end{align}

\begin{lemma}
There exists $c_0 > 0$ such that for any $0 < c < c_0$ and any functions $\varphi$, $\sigma$ satisfying (\ref{phi_cond}) and the following condition: $\sigma(t, s, \dot s) = 0$ for $s = s_1(t)$, $\dot s = \dot s_1(t)$ and for $s = s_2(t)$, $\dot s = \dot s_2(t)$, there is a solution of (\ref{eq9}) satisfying $s(t) \in (s_1(t), s_2(t))$ for all $t > 0$.
\end{lemma}
\begin{proof}
By $m_1$ and $m_2$ we denote the following values
$$
m_1 = \max |\xi \eta' - \eta \xi'|, \quad m_2 = \max |\xi\xi' + \eta\eta'|.
$$
Let $c_0$ be such a number that for a any $0 < c < c_0$, we have
\begin{align}
\label{eq11}
|c m_1 - c^2 m_2| < 1.
\end{align}
Let us consider a subset $W$ of the extended phase space:
$$
W = \{ t, s, \dot s \colon t \geqslant 0, s_1(t) \leqslant s \leqslant s_2(t), \dot s \in \mathbb{R} \}.
$$
Note that solution $s(t)$ of (\ref{eq9}) is tangent to $W$ iff either $s(t) = s_1(t)$ and $\dot s(t) = \dot s_1(t)$ or $s(t) = s_2(t)$ and $\dot s(t) = \dot s_2(t)$. From (\ref{eq9}) and (\ref{eq11}) we have that these solutions are externally tangent to $W$. Therefore, if for some solution $s(t)$ we have $s(t) \in (s_1(t), s_2(t))$ for $t \in [0, t')$ and $s(t') = s_1(t')$, then $\dot s(t') < \dot s_1(t')$. Similarly, from $s(t') = s_2(t')$ we obtain $\dot s(t') < \dot s_2(t')$.

Let us now consider a path $\Gamma \colon [0,1] \to W $ in the plane $t = 0$ such that $\Gamma(0) = (0, s_1(0), a)$ and $\Gamma(1) = (0, s_2(0), b)$, where $a < \dot s_1(0)$ and $b > \dot s_2(0)$.

Suppose that every solution starting at $\Gamma$ leaves $W$. Then it is possible to construct a continuous map between $\Gamma$ and $\partial W$. The boundary $\partial W$ is not connected. This contradiction proves the lemma.
\end{proof}

Using Lemma 3.3 similarly to Theorem 3.1 the following can be proved
\begin{theorem}
Suppose that $g(t)$ is periodic and $\bar g = 0$. Then there exist $c$ such that, for any $\varphi(t)$ satisfying (\ref{eq11}), there exists $\lambda_0$ such that for any $\lambda > \lambda_0$, equation (\ref{eq12}) has a solution $s(t)$ satisfying $s(t) \in (s_1(t), s_2(t))$ for all $t > 0$.
\begin{align}
    \label{eq12}
    \ddot s + \ddot\varphi(\xi\eta' - \eta\xi') - \dot\varphi^2(\xi\xi' + \eta\eta') +  (1 + g(\lambda t))(\xi'\sin\varphi + \eta' \cos\varphi) = 0.
\end{align}
\end{theorem}

It is also possible to prove the existence of periodic solutions in the above systems. First, let us consider the following auxiliary result.

\begin{theorem}
Let us have a boundary value problem
\begin{align}
    \label{eq13}
    u'' = f(t, u), \quad u(0) = u(T), \quad u'(0) = u'(T).
\end{align}
Here $f \colon \mathbb{R} \times \mathbb{R} \to \mathbb{R}$ is $T$-periodic in $t$. If there exist two $T$-periodic functions $\alpha(t)$ and $\beta(t)$ such that for all $t$
\begin{enumerate}
    \item $\alpha(t) \leqslant \beta(t)$,
    \item $\alpha''(t) \geqslant f(t, \alpha(t))$,
    \item $\beta''(t) \leqslant f(t, \beta(t))$.
\end{enumerate}
Then problem (\ref{eq13}) has a solution $u(t)$ such that $\alpha(t) \leqslant u(t) \leqslant \beta(t)$ for all $t$.
\end{theorem}
From this result we immediately  have:
\begin{lemma}
Suppose that function $f(t)$ in (\ref{eq5}) is $T$-periodic. Then there exist a $T$-periodic solution $x(t)$ of (\ref{eq5}) such that $0 < x(t) < \pi$ for all $T$.
\end{lemma}
\begin{lemma}
There exists $c > 0$ such that for any $T$-periodic function $\varphi(t)$ satisfying (\ref{eq8}) there exists a $T$-periodic solution $s(t)$ and for all $t$
$$
s_1(t) < s(t) < s_2(t).
$$
\end{lemma}
From Lemma 3.6, similarly to the proof of Theorem 3.1, the following can be proved
\begin{theorem}
Suppose that functions $f(t)$ and $g(t)$ in (\ref{eq6}) are $T$-periodic and $\bar g = 0$. Then there exists $\lambda_0$ such that for any $\lambda \geqslant \lambda_0$, $\lambda \in \mathbb{N}$, equation (\ref{eq6}) has a $T$-periodic solution $x(t)$ and $x(t) \in (0, \pi)$ for all $t$.
\end{theorem}
We omit the proof of this result since it almost verbatim identical to the proof of Theorem 3.1 and present the proof to the following
\begin{theorem}
Suppose that function $g(t)$ in (\ref{eq9}) is $T$-periodic and $\bar g = 0$. Then there exists $c$ such that for any $T$-periodic $\varphi(t)$ satisfying (\ref{eq11}), there exists $\lambda_0$ such that for any $\lambda \geqslant \lambda_0$, $\lambda \in \mathbb{N}$, equation (\ref{eq12}) has a $T$-periodic solution $s(t)$ and $s(t) \in (s_1(t), s_2(t))$ for all $t$.
\end{theorem}
\begin{proof}
The general scheme of the proof coincides with the one for Theorem 3.1. First, from equation (\ref{eq8}) it follows that there exist $\delta, \varepsilon, \tau > 0$ such that any solution $\tilde{s}(t)$ of (\ref{eq8}) starting at $\tilde{s}(t_0) \in [s_1(t_0), s_1(t_0) + \delta]$, $\dot{\tilde{s}}(t_0) \in [-\delta + \dot s_1(t_0), \delta + s_1(t_0)]$ satisfies the following condition. For some $t' \in [0, \tau]$ $\tilde s(t_0 + t') = s_1(t_0 + t') - \varepsilon$. Similar condition holds for $\tilde{s}(t_0) \in [s_2(t_0) - \delta, s_2(t_0)]$, $\dot{\tilde{s}}(t_0) \in [-\delta + \dot s_2(t_0), \delta + s_2(t_0)]$.

Now consider a modified equation (\ref{eq9}). Here again $\sigma(s, \dot s)$ is a `step' smooth function with bounded derivatives. However, now this function depends on curves $s_1(t)$ and $s_2(t)$ (i.e., since the form of the curve is fixed, on the law of rotation $\varphi(t)$):
$$
\sigma(t, s, \dot s)=
\left\{
		\begin{array}{lll}
			0 & \mbox{if }& s \in [s_1(t) - \frac{\delta}{2}, s_1(t) + \frac{\delta}{2}] \mbox{ and } \dot s \in [\dot s_1(t) - \frac{\delta}{2}, \dot s_1(t) + \frac{\delta}{2}]\\
			&&\mbox{or } \\
			 & \mbox{ }& s \in [s_2(t) - \frac{\delta}{2}, s_2(t) + \frac{\delta}{2}] \mbox{ and } \dot s \in [\dot s_2(t) - \frac{\delta}{2}, \dot s_2(t) + \frac{\delta}{2}],\\
			1 & \mbox{if }& s \not\in [s_1(t) - \delta, s_1(t) + \delta] \mbox{ or } \dot s \in [\dot s_1(t) - \delta, \dot s_1(t) + \delta]\\
			&&\mbox{and } \\
			 & \mbox{ }& s \not\in [s_2(t) - \delta, s_2(t) + \delta] \mbox{ or } \dot s \in [\dot s_2(t) - \delta, \dot s_2(t) + \delta].
		\end{array}
	\right.
$$
And again suppose that $\sigma(t,s,\dot s)$ is smooth elsewhere. From Lemma 3.7, for any $\lambda$, there exists a $T$-periodic solution $s(t)$ of the modified equation. Suppose that for some $t_0$ we have ${s}(t_0) \in [s_1(t_0), s_1(t_0) + \delta]$, $\dot{{s}}(t_0) \in [-\delta + \dot s_1(t_0), \delta + s_1(t_0)]$ or ${s}(t_0) \in [s_2(t_0) - \delta, s_2(t_0)]$, $\dot{{s}}(t_0) \in [-\delta + \dot s_2(t_0), \delta + s_2(t_0)]$. From Theorem 2.1 and the choice of $\delta$, we have that $s(t_0+t') \not\in (s_1(t_0+t'), s_2(t_0+t'))$ for some $t'$. The contradiction proves the theorem.
\end{proof}
\section{Conclusion}
In \cite{bogolyubov1961asymptotic} N.N. Bogolyubov and Y. A. Mitropolskij wrote: `\textit{One can, for instance, try to find conditions under which the difference between the exact solution and its asymptotic approximation, for small values of the parameter, becomes arbitrarily small on an arbitrarily long, yet finite, time interval. It is also possible to consider far more difficult problems trying to find a correspondence between such properties of the exact and asymptotic solutions that depends on their behaviour on the infinite time interval.}'

In other words, the problem of averaging on the infinite time interval is more complicated and can hardly be solved in a usual way. Some results on the matter can be found in \cite{sanders2007averaging,burd2007method}. For instance, some results on the averaging for the infinite time interval can be obtained in the presence of an asymptotically stable solution.

In our work we present a novel approach to this type of problems. We do not consider a small vicinity of an exact solution, but consider a relatively large subset of the extended phase space and we suppose that the flow of the unperturbed (averaged) system does not have trajectories internally tangent to the boundary of our set. Then it is possible to show that, for the original non-autonomous system, there exists a solution that always remains in the considered subset. A solution with the same property will also exist for the averaged system. Moreover, we show that, provided the perturbation is periodic, there will be two periodic solution for the corresponding systems. We cannot say that these solutions are close, but, at least, we can say that always stays in some known region.

In the paper we considered only two simple mechanical systems, yet we believe that this approach can be applied in various systems and can be generalized in different ways.


\section*{References}

\bibliographystyle{model1-num-names}
\bibliography{sample}







\end{document}